\documentclass[11pt, a4paper]{amsart}

\usepackage{amsfonts,amsmath,amssymb, amscd,fullpage}
\usepackage{stmaryrd}
\usepackage[all]{xy}

\newtheorem{theorem}{Theorem}[section]
\newtheorem{lemma}[theorem]{Lemma}
\newtheorem{definition}[theorem]{Definition}
\newtheorem{proposition}[theorem]{Proposition}
\newtheorem{corollary}[theorem]{Corollary}

\theoremstyle{definition}
\newtheorem{rem}[theorem]{Remark}

\newtheorem{question}[theorem]{Question}

\newtheorem{examples}[theorem]{Examples}

\newcommand\pf{\begin{proof}}
\newcommand\epf{\end{proof}}

\newcommand\C{\mathbb{C}}

\newcommand\B{\mathcal{B}}

\newcommand\yd{\mathcal{YD}}

\newcommand\ext{\mathrm{Ext}}
\newcommand\co{\operatorname{co}}

\DeclareMathOperator{\Hom}{Hom}

\DeclareMathOperator{\id}{id}

\DeclareMathOperator{\GL}{GL}
\DeclareMathOperator{\SL}{SL}

\numberwithin{equation}{section}

\title{Gerstenhaber-Schack and Hochschild cohomologies  of  Hopf algebras}

\author{Julien Bichon}
\address{
Laboratoire de Math\'ematiques,
Universit\'e Blaise Pascal,
Campus universitaire des C\'ezeaux,
3 place Vasarely,
63178~Aubi\`ere Cedex, France}
\email{Julien.Bichon@math.univ-bpclermont.fr}


\begin{document}

\begin{abstract}
We show that the Gerstenhaber-Schack cohomology of a  Hopf algebra determines its Hochschild cohomology, and in particular its Gerstenhaber-Schack cohomological dimension bounds its Hochschild cohomological dimension, with equality of the dimensions when the Hopf algebra is cosemisimple of Kac type.
Together with some general considerations on free Yetter-Drinfeld modules over adjoint Hopf subalgebras and the monoidal invariance of Gerstenhaber-Schack cohomology, this is used to show that both Gerstenhaber-Schack and Hochschild cohomological dimensions of the coordinate algebra of the quantum permutation group are $3$. 
\end{abstract}

\maketitle

\section{Introduction}

We study homological properties of Hopf algebras by using Yetter-Drinfeld modules and tensor category techniques. We are especially interested in the following question: 

\begin{question}\label{ques}
If $A$ and $B$ are Hopf algebras having equivalent tensor categories of comodules, how are their Hochschild cohomologies related? In particular do $A$ and $B$ have the same cohomological dimension? 
\end{question}

We have seen in \cite{bic} that the Hochschild cohomologies of two such Hopf algebras $A$ and $B$ are indeed closely related, using resolutions of the trivial Yetter-Drinfeld module over $A$ (or over $B$) formed by free Yetter-Drinfeld modules. In the present paper we continue this study along the same line of ideas.

Our first remark in view of Question \ref{ques} is that there exists at least a cohomology theory for Hopf algebras that is known to be well-behaved with respect to this situation: Gerstenhaber-Schack cohomology \cite{gs1,gs2}.
Let $A$ be a Hopf algebra and let $M$ be a Hopf bimodule over $A$: the Gerstenhaber-Schack cohomology $H^*_{\rm GS}(A,M)$ of $A$ with coefficients in $M$ \cite{gs2} is defined to be the homology of an explicit bicomplex whose
columns are modeled on the Hochschild complex of the underlying algebra
and rows are modeled on the Cartier complex of the underlying coalgebra.
When $M=A$ is the trivial Hopf bimodule, then $H^*_{\rm GS}(A,A)=:H_b^*(A)$ is known as the bialgebra cohomology of $A$. This cohomology theory, which can also be defined in terms of Yetter-Drinfeld modules, was first introduced in view of applications to deformation theory, and  has been used as a key tool in the proof of several fundamental results on finite-dimensional Hopf algebras \cite{ste2,eg}.

If $A$ and $B$ are Hopf algebras having equivalent tensor categories of comodules, then there exists a tensor equivalence $F: {\!_A^A\mathcal M}_A^A \rightarrow {\!_B^B\mathcal M}_B^B$ between their categories of Hopf bimodules such that for any Hopf bimodule $M$ over $A$, we have $H_{\rm GS}^*(A,M)\simeq H_{\rm GS}^*(B,F(M))$, and in particular $H^*_b(A) \simeq H^*_b(B)$ and ${\rm cd}_{\rm GS}(A)={\rm cd}_{\rm GS}(B)$ (where ${\rm cd}_{\rm GS}$ denotes the Gerstenhaber-Schack cohomological dimension, defined in the obvious way, see Section 5). We call these properties the monoidal invariance of Gerstenhaber-Schack cohomology.

Going back to Question \ref{ques}, the next question is to study how Hochschild and Gerstenhaber-Schack cohomologies are related. We show that the Gerstenhaber-Schack cohomology of a  Hopf algebra determines its Hochschild cohomology. More precisely, we show that if $A$ is a  Hopf algebra, then there exists a functor
$G: {\!_A\mathcal M}_A \rightarrow {\!_A^A\mathcal M}_A^A$ such that for any $A$-bimodule $M$, we have
$$H^*(A,M)\simeq H_{\rm GS}^*(A,G(M))$$
In particular we have ${\rm cd}(A)\leq {\rm cd}_{\rm GS}(A)$.
Then if $A$ and $B$ are  Hopf algebras as in Question \ref{ques}, combining this with the monoidal invariance of  Gerstenhaber-Schack cohomology, we get the existence of two functors $F_1 : {\!_A\mathcal M}_A \rightarrow {\!_B^B\mathcal M}_B^B$ and $F_2 : {\!_B\mathcal M}_B \rightarrow {\!_A^A\mathcal M}_A^A$
such that for any $A$-bimodule $M$ and any  $B$-bimodule $N$, we have
$$H^*(A,M)\simeq H_{\rm GS}^*(B,F_1(M)) \ {\rm and} \ H^*(B,N)\simeq H_{\rm GS}^*(A,F_2(N))$$
In particular
$$\max({\rm cd}(A),{\rm cd}(B))\leq {\rm cd}_{\rm GS}(A)={\rm cd}_{\rm GS}(B)$$
These isomorphisms and this inequality thus provide partial answers to Question \ref{ques}. They lead to the following new question:

\begin{question}\label{ques2}
  Is it true that ${\rm cd}(A)={\rm cd}_{\rm GS}(A)$ for any  Hopf algebra $A$ over a field of characteristic zero? Is it true at least if $A$ is assumed to be cosemisimple?
\end{question}

A positive answer would give the monoidal invariance of cohomological dimension and fully answer the last part of Question \ref{ques}, and would also be a natural infinite-dimensional generalization of a famous result by Larson-Radford \cite{lr}, which states that, in characteristic zero, a finite-dimensional cosemisimple Hopf algebra is semisimple. See Remark \ref{rem:lr}. 

We provide (Corollary \ref{cor:partansques2}) a partial positive answer to Question \ref{ques2} in the case where $A$ is cosemisimple of Kac type (the square of the antipode is the identity), and in turn this gives a partial positive answer to Question \ref{ques} (see Corollary \ref{cor:partansques}).

We then apply our general considerations to quantum symmetry Hopf algebras, which were the first motivation for this work. Let $(R, \varphi)$ semisimple measured algebra  of dimension $\geq 4$, and let $A_{\rm aut}(R, \varphi)$ be its quantum symmetry Hopf algebra \cite{wan98,bi00}. 
We compute, in the cosemisimple case, the bialgebra cohomology of $A_{\rm aut}(R,\varphi)$,  and we show that ${\rm cd}(A_{\rm aut}(R,\varphi))\leq {\rm cd}_{\rm GS}(A_{\rm aut}(R,\varphi))=3$, with equality if $\varphi$ is a trace.
These results include  in particular the coordinate algebra of Wang's quantum permutation group $S_n^+$ \cite{wan98}.


As a last comment to further motivate the use of Gerstenhaber-Schack cohomology as an appropriate cohomology theory for Hopf algebras (apart from its use to get information on Hochschild cohomology itself), we would like to point out that, in the examples computed so far, it also has the merit to avoid the ``dimension drop'' phenomenon
usually encountered for quantum algebras (see \cite{hk,hk2}): the canonical choice of coefficients (the trivial Hopf bimodule) is the good one to get the cohomological dimension. It would be interesting to know if this can be further generalized.

The paper is organized as follows. Section 2 consists of preliminaries. In Section 3 we discuss the cohomological dimension of a Hopf subalgebra and the sub-additivity of the cohomological dimension under extensions.
Section 4 is devoted to Yetter-Drinfeld modules: we recall the concept of free (resp. co-free) Yetter Drinfeld module and we introduce the notion of relative projective (resp. injective) Yetter-Drinfeld module, which corresponds, via the  tensor equivalence between Yetter-Drinfeld modules and Hopf bimodules \cite{sc94}, to the notion of relative projective (resp. injective) Hopf bimodule considered in \cite{shst}. We show that relative projective (resp. injective) Yetter-Drinfeld modules are precisely the direct summands of free (resp. co-free) Yetter-Drinfeld modules. This section also contains some considerations on free Yetter-Drinfeld modules over adjoint Hopf subalgebras. In Section 5, after having recalled some basic facts on Gerstenhaber-Schack cohomology,  we provide an explicit complex that computes the Gerstenhaber-Schack cohomology $H^*_{\rm GS}(A,V)$, if $A$ is cosemisimple or if the Yetter-Drinfeld module $V$ is relative injective, using results from \cite{shst} in this last case (see Proposition \ref{compgs}).
We then show that Gerstenhaber-Schack cohomology determines Hochschild cohomology, and show that Question \ref{ques2} has a positive answer in the case of cosemisimple Hopf algebras of Kac type.
In Section 6 we study the examples mentioned earlier in the introduction.
In Section 7 we discuss the Gerstenhaber-Schack cohomological dimension in the setting of  Hopf algebras having a projection. 

\section{Preliminaries}

In this preliminary section we fix some notation, we recall some basic definitions and facts on   the Hochschild cohomology of a Hopf algebra, and we discuss exact sequences of Hopf algebras.

\subsection{Notations and conventions}
We work over $\mathbb C$ (or over any algebraically closed field of characteristic zero). This assumption does not affect any of the theoretical results in the paper, but is important for the examples we consider.
We assume that the reader is familiar with the theory of Hopf algebras and their tensor categories of comodules, as e.g. in \cite{kas,ks,mon}.
If $A$ is a Hopf algebra, as usual, $\Delta$, $\varepsilon$ and $S$ stand respectively for the comultiplication, counit and antipode of $A$. We use Sweedler's notations in the standard way. The category of right $A$-comodules is denoted $\mathcal M^A$, the category of right $A$-modules is denoted $\mathcal M_A$, etc...  
The trivial (right) $A$-module is denoted $\C_\varepsilon$.
The set of $A$-module morphisms (resp. $A$-comodule morphisms) between two $A$-modules (resp. two $A$-comodules) $V$ and $W$ is denoted ${\rm Hom}_A(V,W)$ (resp. ${\rm Hom}^A(V,W)$).

\subsection{Hochschild cohomology of a Hopf algebra} If $A$ is an algebra and $M$ is an $A$-bimodule, then $H^*(A,M)$ denotes, as usual, the  Hochschild cohomology  of $A$ with coefficients in $M$. See e.g. \cite{wei}.

\begin{definition}
 The Hochschild cohomological dimension of an algebra $A$ is
defined to be 
$${\rm cd}(A)= {\rm sup}\{n : \ H^n(A, M) \not=0 \ {\rm for} \ {\rm some} \ A-{\rm bimodule} \ M\}\in \mathbb N \cup \{\infty\}$$   
\end{definition}

As noted by several authors (see \cite{ft}, \cite{gk}, \cite{hk}, \cite{bz}, \cite{cht}, \cite{bic}),
the Hochschild cohomology of a Hopf algebra can be described by using a suitable $\ext$ functor on the category of left or right $A$-modules.
Indeed, if $A$ is a Hopf algebra and $M$ is an $A$-bimodule, we have
$$H^*(A,M) \simeq \ext^*_A(\C_\varepsilon, M')$$ 
where the above $\ext$ is in the category of right $A$-modules and $M'$ is $M$ equipped with the right $A$-module structure given by $x \leftarrow a = S(a_{(1)})\cdot x \cdot a_{(2)}$.

This leads to the following description of the cohomological dimension of a Hopf algebra.

\begin{proposition}\label{cdhopf}
 Let $A$ be a Hopf algebra. We have
\begin{align*}
 {\rm cd}(A)&= {\rm sup}\{n : \ \ext^n_A(\mathbb C_\varepsilon, M) \not=0 \ {\rm for} \ {\rm some} \ A-{\rm module} \ M\} \\
& = {\rm inf}\{n: \ \ext^i_A(\mathbb C_\varepsilon, -) =0 \ {\rm for} \ i>n\} \\
& = {\rm inf}\{n: \ \mathbb C_\varepsilon \ {\rm admits} \ {\rm a} \ {\rm projective} \  {\rm resolution} \ {\rm of} \ {\rm length} \ n\}   
\end{align*}
\end{proposition}

\begin{proof}
The previous isomorphism ensures that 
$${\rm cd}(A)\leq {\rm sup}\{n : \ \ext^n_A(\mathbb C_\varepsilon, M) \not=0 \ {\rm for} \ {\rm some} \ A-{\rm module} \ M\}$$
If $V$ is a right $A$-module, let $_\varepsilon V$ be the $A$-bimodule whose right structure is that of $V$ and whose left structure is trivial, i.e. given by $\varepsilon$. Then $(_\varepsilon V)'=V$, hence the converse inequality holds, and the first equality in the statement is proved, as well as the second one.
The last  one is shown similarly as in the case of group cohomology, see e.g. \cite[Chapter VIII, Lemma 2.1]{br}. 
\end{proof}

\begin{examples}
 \begin{enumerate}
  \item If $G$ is a linear algebraic group, with coordinate algebra $\mathcal O(G)$, it is well-known that ${\rm cd}(\mathcal O(G))= \dim G$.
\item If $\Gamma$ is a (discrete) group, then ${\rm cd}(\C\Gamma)={\rm cd}_\C(\Gamma)$, the cohomological dimension of $\Gamma$ with coefficients $\C$. We have ${\rm cd}(\C\Gamma)=0$  if and only if $\Gamma$ is finite, and if $\Gamma$ is finitely generated, then ${\rm cd}(\C\Gamma)=1$ if and only if $\Gamma$ contains a free normal subgroup of finite index, see \cite{dun}.
\item If $A$ is a finite-dimensional Hopf algebra, then either ${\rm cd}(A)=0$ (when $A$ is semisimple) or ${\rm cd}(A)=\infty$, a finite-dimensional Hopf algebra being Frobenius and hence self-injective.
 \end{enumerate}
\end{examples}

\subsection{Exact sequences of Hopf algebras}
A sequence  of Hopf algebra maps
\begin{equation*}\C \to B \overset{i}\to A \overset{p}\to L \to
\C\end{equation*} is said to be exact \cite{ad} if the following
conditions hold:
\begin{enumerate}\item $i$ is injective and $p$ is surjective,
\item $\ker p =Ai(B)^+ =i(B)^+A$, where $i(B)^+=i(B)\cap{\rm Ker}(\varepsilon)$,
\item $i(B) = A^{\co L} = \{ a \in A:\, (\id \otimes p)\Delta(a) = a \otimes 1
\} = {^{\co L}A} = \{ a \in A:\, (p \otimes \id)\Delta(a) = 1 \otimes a
\}$. \end{enumerate}
Note that condition (2) implies  $pi= \varepsilon 1$.

\begin{proposition}\label{refoexact}
Let \begin{equation*}\C \to B \overset{i}\to A \overset{p}\to L \to \C
    \end{equation*}
be a sequence of Hopf algebra maps where $i$ is injective, $p$ is surjective and $pi$ = $\varepsilon 1$.
Assume that the antipode of $A$ is bijective. Consider the following three assertions.
\begin{enumerate}
 \item $A$ is faithfully flat as a right $B$-module and ${\rm Ker}(p) =Ai(B)^+ =i(B)^+A$.
\item   $^{\co L}A=A^{\co L}=i(B)$ and $p$ is left or right faithfully coflat.
\item The sequence is exact.
\end{enumerate}
Then we have $(1) \Rightarrow (3)$ and $(2)\Rightarrow (3)$, and if $(3)$ holds, then we have $(1)\iff (2)$.
\end{proposition}

An exact sequence satisfying $(1)$ and $(2)$ is called strict \cite{schneider}. Note that $p$ is automatically faithfully coflat if $L$ is cosemisimple.

That $(1) \Rightarrow (3)$ holds is well-known (see \cite[Proposition 1.2.4]{ad},  \cite[Lemma 1.3]{schn},  \cite[Proposition 3.4.3]{mon}, or more generally \cite[Theorem 1]{tak}). Also that $(1)\iff (2)$ if $(3)$ holds is known, see \cite[Corollary 1.8]{schn}.
I wish to thank the referee for pointing out that  $(2) \Rightarrow (3)$ follows from \cite[Theorem 2]{tak}, combined with  \cite[Remark 1.3]{musc}.

\section{Cohomological dimension of a Hopf subalgebra}

In this section we discuss the behavior of cohomological dimension when passing to a Hopf subalgebra, which, under mild  assumptions, is similar to the group cohomology case.

\begin{proposition}\label{sub}
  Let $B \subset A$ be a Hopf subalgebra. Assume that one of the following conditions holds.
\begin{enumerate}
\item $A$ is projective as a right $B$-module.
\item The antipode of $A$ is bijective and $A$ is faithfully flat as a right $B$-module.
\item $A$ is cosemisimple.
\item The exists a Hopf algebra map $\pi : A \rightarrow B$ such that $\pi_{|B}={\rm id}_B$.
\item The antipode of $A$ is bijective and $B$ is commutative.
\end{enumerate}
Then ${\rm cd}(B) \leq {\rm cd}(A)$.
\end{proposition}

\begin{proof}
If $A$ is projective as a right $B$-module, any projective right $A$-module is projective as a right $B$-module, thus a resolution of length $n$ of $\mathbb C_\varepsilon$ in $\mathcal M_A$ yields a resolution of length $n$ in $\mathcal M_B$, and thus Proposition \ref{cdhopf} ensures that ${\rm cd}(B) \leq {\rm cd}(A)$.
Assuming (2), Corollary 1.8 in \cite{schn} yields that $A$ is projective as a right $B$-module, and we conclude by (1). If we assume that $A$ is cosemisimple, then its antipode is bijective and by \cite{chi} $A$ is faithfully flat as a right $B$-module, and we conclude by (2). If we assume (4), then $A$ is free as a right $B$-module, see \cite{rad85} (we will come back to this situation in Section 7), thus we conclude by (1). If $B$ is commutative, then $A$ is faithfully flat over $B$ by Proposition 3.12 in \cite{ag}, and again we conclude by (2).
\end{proof}

The following result is the generalization of the sub-additivity of cohomological dimension under extensions (see e.g. Proposition 2.4 in \cite{br}) 
with essentially the same proof, using Stefan's spectral sequence \cite{ste} as the natural generalization of the Hochschild-Serre spectral sequence.

\begin{proposition}\label{subadd}
  Let \begin{equation*} \C \longrightarrow B \overset{i}\longrightarrow A \overset{p}\longrightarrow L \longrightarrow \C
    \end{equation*}
be a strict exact sequence of Hopf algebras , and assume that the antipode of $A$ is bijective.
Then we have ${\rm cd}(B)\leq  {\rm cd}(A) \leq {\rm cd}(B) + {\rm cd}(L)$. If moreover $L$ is semisimple, then 
${\rm cd}(B)={\rm cd}(A)$.
\end{proposition}

\begin{proof}
 By \cite[Lemma 1.3]{schn}, (or more generally \cite[Theorem 1]{tak}, see also \cite[Proposition 3.4.3]{mon}), the canonical map
\begin{align*}
 A \otimes_B A &\longrightarrow A \otimes L \\
a \otimes_B a' &\longmapsto aa'_{(1)} \otimes p(a'_{(2)})
\end{align*}
is bijective. Thus $B \subset A$ is an $L$-Galois extension, and $A$ is faithfully flat both as a left and right $B$-module (the antipode of $A$ is bijective). Thus for any $A$-$A$-bimodule $M$ there exists a spectral sequence \cite{ste}
$$E_2^{pq} = H^p(L;H^q(B,M))\Rightarrow H^{p+q}(A,M)$$
The spectral sequence is concentrated in the rectangle $0\leq p \leq {\rm cd}(L)$, $0 \leq q \leq {\rm cd}(B)$, and it follows that for $i>   {\rm cd}(L) + {\rm cd}(B)$, we have $H^i(A,M)=0$, and this proves the inequality. If $L$ is semisimple, then ${\rm cd}(L)=0$, and hence  ${\rm cd}(B)={\rm cd}(A)$.
\end{proof}




\section{Yetter-Drinfeld modules}

Let $A$ be a Hopf algebra. Recall that a (right-right) Yetter-Drinfeld
module over $A$ is a right $A$-comodule and right $A$-module $V$
satisfying the condition, $\forall v \in V$, $\forall a \in A$, 
$$(v \leftarrow a)_{(0)} \otimes  (v \leftarrow a)_{(1)} =
v_{(0)} \leftarrow a_{(2)} \otimes S(a_{(1)}) v_{(1)} a_{(3)}$$
The category of Yetter-Drinfeld modules over $A$ is denoted $\yd_A^A$:
the morphisms are the $A$-linear $A$-colinear maps.
Endowed with the usual tensor product of 
modules and comodules, it is a tensor category, with unit the trivial Yetter-Drinfeld module, denoted $\C$.

An important example of Yetter-Drinfeld module is the right coadjoint Yetter-Drinfeld
module $A_{\rm coad}$: as a right $A$-module $A_{\rm coad}=A$ and the right $A$-comodule structure
is defined by 
$${\rm ad}_r(a) = a_{(2)} \otimes S(a_{(1)})a_{(3)}, \forall a \in A$$
The coadjoint Yetter-Drinfeld module has a natural generalization, discussed in the next subsection.

\subsection{Free and co-free Yetter-Drinfeld modules} We now discuss some important constructions of Yetter-Drinfeld modules (left-right versions of these constructions were first given in \cite{camizh97}, see \cite{shst} as well, in the context of Hopf bimodules).

Let $V$ be a right $A$-comodule. The Yetter-Drinfeld module $V \boxtimes A$ is defined as follows \cite{bic}. 
As a vector space $V \boxtimes A=V \otimes A$, the right module structure is given by multiplication on the right, and the right coaction $\alpha_{V \boxtimes A}$ is defined by
$$ \alpha_{V \boxtimes A}( v \otimes a)=v_{(0)}  \otimes a_{(2)} \otimes S(a_{(1)})v_{(1)}a_{(3)}$$
Note that $A_{\rm coad} = \mathbb C \boxtimes A$.

A Yetter-Drinfeld module is said to be free if it is isomorphic to $V\boxtimes A$ for some comodule $V$.

The construction of the free Yetter-Drinfeld module on a comodule yields a functor $ L = - \boxtimes A: \mathcal M^A \longrightarrow \yd_A^A$ which is left adjoint to
  the forgetful functor $ R :   \yd_A^A \longrightarrow \mathcal M^A$ . Indeed 
we have natural isomorphisms
\begin{align*}
\Hom^A(V, R(X))  &\longrightarrow  \Hom_{\yd_A^A}(V \boxtimes A, X)
  \\
 f &\longmapsto \tilde{f}, \ \tilde{f}(v\otimes a)=f(v)\leftarrow a
 \end{align*}
for any $A$-comodule $V$ and any Yetter-Drinfeld module $X$.

Now let $M$ be a right $A$-module. The Yetter-drinfeld module $M \#A$ is defined as follows: the underlying vector space is $M \otimes A$, the right coaction is  ${\rm id}_M \otimes \Delta$, while the right action is given by
$$(x \otimes a) \leftarrow b= x\cdot b_{(2)} \otimes S(b_{(1)})ab_{(3)}$$ 
The Yetter-Drinfeld module $\C\#A$ is the adjoint Yetter-Drinfeld module, denoted $A_{\rm ad}$. 

A Yetter-Drinfeld module will be said to be co-free if it is isomorphic to $M\# A$ for some module $M$.
The construction of the co-free Yetter-Drinfeld module on a module yields a functor $ L = - \# A: \mathcal M_A \longrightarrow \yd_A^A$ which is right adjoint to
  the forgetful functor $L :   \yd_A^A \longrightarrow \mathcal M_A$. Indeed 
we have natural isomorphisms
\begin{align*}
\Hom_{\yd_A^A}(X, M \# A) & \longrightarrow \Hom_A(L(X), M)  
  \\
 f &\longmapsto ({\rm id}_M \otimes \varepsilon)f, 
 \end{align*}
for any $A$-module $M$ and any Yetter-Drinfeld module $X$.

\subsection{Relative projective and relative injective Yetter-Drinfeld modules} We will use the following notions.

\begin{definition}
Let $V$ be a Yetter-Drinfeld module over $A$. Then $V$
 is said to be relative projective if the functor 
${\rm Hom}_{\yd_A^A}(V,-)$ transforms exact sequences of Yetter-Drinfeld modules that split as sequences of comodules  to exact sequences of vector spaces. 

Similarly $V$ is said to be relative injective if the functor 
${\rm Hom}_{\yd_A^A}(-,V)$ transforms exact sequences of Yetter-Drinfeld modules that split as sequences of modules  to exact sequences of vector spaces. 
\end{definition}

Relative projective Yetter-Drinfeld modules have the following characterization. 

\begin{proposition}\label{caracrelproj}
 Let $P$ be a Yetter-Drinfeld module over $A$. The following assertions are equivalent.
\begin{enumerate}
 \item $P$ is relative projective.
\item Any surjective morphism of Yetter-Drinfeld modules $f : M \rightarrow P$ that admits a section which is a map of comodules admits a section which is a morphism of Yetter-Drinfeld modules.
\item $P$ is a direct summand of a free Yetter-Drinfeld module.
\end{enumerate}
If $A$ is cosemisimple, these conditions are equivalent to $P$ being a projective object of $\yd_A^A$.
\end{proposition}

\begin{proof}
The proof of  (1)$\Rightarrow$(2) is similar to the usual one for modules. Assume (2), and consider the surjective Yetter-Drinfeld module morphism $R(P)\boxtimes A \rightarrow P$, $x \otimes a \mapsto x\leftarrow a$.
The map $P \rightarrow  R(P)\boxtimes A$, $x \mapsto x \otimes 1$ is an $A$-colinear section, so by (2) $P$ is indeed, as a Yetter-Drinfeld module,  a direct summand of $R(P)\boxtimes A$.

Assume now that $P$ is free, i.e. $P=V\boxtimes A$ for some comodule $V$, and let 
$$0 \to M \overset{i}\to N \overset{p}\to Q \to 0$$
be an exact sequence of Yetter-Drinfeld modules that splits as a sequence of comodules.
The sequence 
 $$0 \to {\rm Hom}_{\yd_A^A}(P,M) \overset{i-}\longrightarrow {\rm Hom}_{\yd_A^A}(P,N) \overset{p -}\longrightarrow {\rm Hom}_{\yd_A^A}(P,Q)$$
is exact and we have to show the surjectivity of the map on the right.
Let $s : Q \rightarrow N$ be a morphism of comodules such that $ps={\rm id}_Q$.
Let $\varphi \in {\rm Hom}_{\yd_A^A}(V \boxtimes A, Q)$, and let $\varphi_0 : V \rightarrow Q$ be defined by $\varphi_0(v)=\varphi (v \otimes 1)$: $\varphi_0$ is a map of comodules, and so is $s\varphi_0$.
Considering now $\widetilde{s\varphi_0} \in {\rm Hom}_{\yd_A^A}(V \boxtimes A,N)$, we have $p\widetilde{s\varphi_0}= \varphi$, which gives the expected surjectivity result.
Now if $V\boxtimes A \simeq P \oplus M$ as Yetter-Drinfeld modules, then ${\rm Hom}_{\yd_A^A}(V\boxtimes A,-) \simeq {\rm Hom}_{\yd_A^A}(P,-) \oplus {\rm Hom}_{\yd_A^A}(M,-)$, and the usual argument for projective modules work to conclude that $P$ is relative projective.

It is clear that a projective Yetter-Drinfeld module is relative projective, and if $A$ is cosemisimple, a free Yetter-Drinfeld module is a projective object in $\yd_A^A$ (Proposition 3.3 in \cite{bic}), hence a direct summand of a free Yetter-Drinfeld module is projective, and so is a relative projective Yetter-Drinfeld module.  
\end{proof}

Similarly, relative injective Yetter-Drinfeld modules are characterized as follows. The proof is similar to the one of the previous result, and is left to the reader.

\begin{proposition}\label{caracrelinj}
 Let $I$ be a Yetter-Drinfeld module over $A$. The following assertions are equivalent.
\begin{enumerate}
 \item $I$ is relative injective.
\item Any injective morphism of Yetter-Drinfeld modules $f : I \rightarrow M$ that admits a section which is a map of modules admits a section which is a morphism of Yetter-Drinfeld modules.
\item $P$ is a direct summand of a co-free Yetter-Drinfeld module.
\end{enumerate}
\end{proposition}

\subsection{Yetter-Drinfeld modules and Hopf bimodules}
In this subsection we briefly recall the category equivalence between Yetter-Drinfeld modules and Hopf bimodules \cite{sc94}, and check that the notion of relative projective objects (resp. relative injective objects) for Yetter-Drinfeld modules corresponds to that for Hopf bimodules considered in \cite{shst}.

First recall that a Hopf bimodule over $A$ is an $A$-bimodule and $A$-bicomodule $M$ whose respective left and right coactions $\lambda : M \rightarrow A \otimes M$ and $\rho : M \rightarrow M \otimes A$ are $A$-bimodule maps.
The category of Hopf bimodules over $A$, whose morphisms are the bimodule and bicomodule maps, is denoted 
${\!_A^A\mathcal M}_A^A$.

If $M$ is Hopf bimodule over $A$, then ${^{{\rm co}A}\!M}=\{x \in M \ | \ \lambda(x)=1 \otimes x\}$ is a right subcomodule of $M$, and inherits a right $A$-module structure given by $x\leftarrow a= S(a_{(1)}).x.a_{(2)}$, making it into a Yetter-Drinfeld module over $A$.
This defines a functor
\begin{align*}
 {\!_A^A\mathcal M}_A^A &\longrightarrow \yd_A^A \\
M & \longmapsto {^{{\rm co}A}\!M}
\end{align*}
Conversely, starting from a Yetter-Drinfeld module $V$, one defines a Hopf bimodule structure on $A \otimes V$ as follows. The bimodule structure is given by
$$a.(b\otimes v).c= abc_{(1)} \otimes (v\leftarrow c_{(2)})$$ 
and the bicomodule structure is given by the following left and right coactions
\begin{align*}
 \lambda : A \otimes V &\longrightarrow A \otimes A \otimes V \quad \quad \quad \ \ \ \ \ \rho : A \otimes V \longrightarrow A \otimes V \otimes A \\
a \otimes v & \longmapsto  a_{(1)} \otimes a_{(2)} \otimes v \quad \quad \quad \quad \quad a \otimes v  \longmapsto a_{(1)} \otimes v_{(0)} \otimes  a_{(2)} v_{(1)}
\end{align*}
If $f : V \longrightarrow W$ is a morphism of Yetter-Drinfeld module, then ${\rm id}_A \otimes f : A \otimes V \rightarrow A \otimes W$ is a morphism of Hopf bimodules, and hence we get a functor 
\begin{align*}
\yd_A^A  &\longrightarrow {\!_A^A\mathcal M}_A^A  \\
V & \longmapsto A\otimes V
\end{align*}
The two functors just defined are quasi-inverse equivalences, see \cite{sc94}.

\begin{lemma}\label{comprelprojinj}
 Relative projective (resp. relative injective) objects in $\yd_A^A$ correspond, via the category equivalence $\yd_A^A \simeq {\!_A^A\mathcal M}_A^A$, to relative projective (resp. relative injective) objects of ${\!_A^A\mathcal M}_A^A $ in the sense of \cite{shst}.
\end{lemma}

\begin{proof}
 Let $M$ be a Hopf bimodule over $A$. That $M$ is relatively projective means that 
the functor 
${\rm Hom}_{{\!_A^A\mathcal M}_A^A}(M,-)$ transforms exact sequences of Hopf bimodules that split as sequences of bicomodules  to exact sequences of vector spaces. The proof of the lemma easily reduces to the following statement.                                                                                                                                                                                                                                 

Let $f: V \to W$ be a surjective  morphism of Yetter-Drinfeld modules, inducing a surjective morphism of  Hopf bimodules ${\rm id}_A \otimes f : A \otimes V \rightarrow A \otimes W$. Then there exists an $A$-comodule section to $f$ if and only if  there exists an $A$-bicomodule section to ${\rm id}_A \otimes f$.

Indeed, if $s : W \rightarrow V$ is $A$-colinear with $fs={\rm id}_W$, then  ${\rm id}_A \otimes s : A \otimes W \rightarrow A \otimes V$ is $A$-bicolinear and is a section to ${\rm id}_A \otimes f$.
Conversely starting with an $A$-bicolinear  map $T : A \otimes W \rightarrow A \otimes V$ with  $({\rm id}_A \otimes f)T= {\rm id}_{A \otimes W}$, then the map $s : W \rightarrow V$ defined by $s(w)= \varepsilon \otimes {\rm id}_V(T(1 \otimes w))$ is $A$-colinear, and satisfies $fs={\rm id}_W$.

Similarly, that $M$ is relatively injective means that 
the functor 
${\rm Hom}_{{\!_A^A\mathcal M}_A^A}(-,M,)$ transforms exact sequences of Hopf bimodules that split as sequences of bimodules  to exact sequences of vector spaces. The proof that this corresponds to the notion of relative injective Yetter-Drinfeld module is left to the reader.
\end{proof}

\subsection{Adjoint Hopf subalgebras}
We now discuss the way to restrict certain  free Yetter-Drinfeld to adjoint Hopf subalgebras.

\begin{proposition}\label{caradj}
 Let $B \subset A$ be a Hopf subalgebra. The following assertions are equivalent.
\begin{enumerate}
 \item For any $a \in A$ and $b \in B$, we have
$$a_{(2)} \otimes S(a_{(1)})ba_{(3)} \in A \otimes B$$
\item For any $B$-comodule $W$, we have $\alpha_{V \boxtimes A}(W \boxtimes A) \subset (W \boxtimes A) \otimes B$ so that $W \boxtimes A$ is an object of $\yd_{B}^B$.
\end{enumerate}
\end{proposition}

\begin{proof}
  $(1) \Rightarrow (2)$ follows from the definition of  $\alpha_{V \boxtimes A}$. Conversely, assuming that $(2)$ holds, take $W=B$ the regular $B$-comodule. Then for any $a \in A$ and $b \in B$, we have
$$b_{(1)} \otimes a_{(2)} \otimes S(a_{(1)})b_{(2)}a_{(3)} \in A \otimes A\otimes B$$
and hence 
$$a_{(2)} \otimes S(a_{(1)})ba_{(3)} = a_{(2)} \otimes S(a_1)\varepsilon(b_{(1)})b_{(2)}a_{(3)}\in A \otimes B$$
Thus $(1)$ holds.
\end{proof}

\begin{definition}
 A Hopf subalgebra $B\subset A$ is said to be adjoint if it satisfies the equivalent conditions of Proposition \ref{caradj}.
\end{definition}

Very often adjoint Hopf subalgebras are obtained in the following way. Recall that a Hopf algebra map $f : A \rightarrow L$ is said to be cocentral if $f(a_{(1)})\otimes a_{(2)}= f(a_{(2)})\otimes a_{(1)}$ for any $a \in A$.

\begin{proposition}
 Let  $B \subset A$ be a Hopf subalgebra.  Assume that there exists a cocentral and surjective Hopf algebra map $p : A \rightarrow L$ such that $B=A^{{\rm co} L}$. Then $B\subset A$ is an adjoint Hopf subalgebra.
Conversely if $B\subset A$ is an adjoint Hopf subalgebra, if $A$ and $B$ have bijective antipodes and if $A$ is faithfully flat as a right $B$-module, then there exists a cocentral surjective Hopf algebra map $p : A \rightarrow L$ such that $B=A^{{\rm co} L}$.
\end{proposition}

\begin{proof}
 Let $a \in A$ and $b \in B$. Since $p(b)=\varepsilon(b)1$, we have, using the cocentrality of $p$,
\begin{align*}{\rm id}_A \otimes {\rm id}_B \otimes p & \left(a_{(2)} \otimes (S(a_{(1)})ba_{(3)})_{(1)} \otimes (S(a_{(1)})ba_{(3)})_{(2)}\right) \\
& = {\rm id}_A \otimes {\rm id}_B \otimes p\left(a_{(3)} \otimes S(a_{(2)})b_{(1)}a_{(4)} \otimes S(a_{(1)})b_{(2)}a_{(5)}\right)\\
&= a_{(3)} \otimes S(a_{(2)})b_{(1)}a_{(4)} \otimes pS(a_{(1)})p(b_{(2)})p(a_{(5)}) \\
& = a_{(3)} \otimes S(a_{(2)})ba_{(4)} \otimes pS(a_{(1)})p(a_{(5)})\\
&= a_{(2)} \otimes S(a_{(1)})ba_{(3)} \otimes1
\end{align*}
Hence $a_{(2)} \otimes S(a_{(1)})ba_{(3)} \in A \otimes A^{{\rm co} L}= A \otimes B$: this shows that $B \subset A$ is adjoint.

Conversely, assume  that $B \subset A$ is adjoint, that $A$ and $B$ have bijective antipodes and that $A$ is faithfully flat as a right $B$-module. Then for any $a \in A$ and $b \in B$, we have
$$S(a_{(1)}) b a_{(2)} =  \varepsilon(a_{(2)}) \varepsilon(b_{(1)})S(a_{(1)})b_{(2)}a_{(3)} \in B$$
It is well-known that this implies $B^+A\subset AB^+$, and hence $AB^+\subset B^+A$ by the bijectivity of the antipodes. It follows that $B^+A$ is a Hopf ideal in $A$, and we denote by $p: A \rightarrow A/B^+A=L$ the canonical Hopf algebra surjection. By construction we have $B \subset A^{{\rm co}L}$, and 
for $b \in B$ we have $p(b)=\varepsilon(b)$. Hence we have for any $a \in A$,
$a \otimes 1= a_{(2)}\otimes p(S(a_{(1)})a_{(3)})$, hence
$$a_{(2)} \otimes p(a_{(1)}) = (1\otimes p(a_{(1)})) (a_{(2)} \otimes 1) =
(1 \otimes p(a_{(1)})(a_{(3)} \otimes p(S(a_{(2)})a_{(4)}))=a_{(1)} \otimes p(a_{(2)})$$
and this shows that $p$ is cocentral. Finally we have $B= A^{{\rm co}L}$ by Corollary 1.8 in \cite{schn}.
\end{proof}

We now discuss a condition that ensures that the restriction of a free Yetter-Drinfeld module to an adjoint Hopf subalgebra as in Proposition \ref{caradj} remains a relative projective Yetter-Drinfeld module.

\begin{proposition}\label{sigma}
  Let $B \subset A$ be a Hopf subalgebra with  $B=A^{{\rm co} L}$ for some cocentral and surjective Hopf algebra map $p : A \rightarrow L$. Assume that there exists a linear map $\sigma : L \rightarrow A$ such that
\begin{enumerate}
 \item $p\sigma= {\rm id}_L$;
\item $\sigma(x)_{(1)} \otimes p(\sigma(x)_{(2)})=\sigma(x_{(1)}) \otimes x_{(2)}$, for any $x \in L$;
\item $\sigma(x)_{(1)} S(\sigma(x)_{(3)})\otimes \sigma(x)_{(2)} = 1_B \otimes \sigma(x)$, for any $x \in L$.
\end{enumerate}
Then for any $B$-comodule $W$, the object $W \boxtimes A \in \yd_B^B$ is relative projective. Such a map $\sigma$ exists if $A$ is cosemisimple.
\end{proposition}

\begin{proof}
 We first claim that for any $a\in A$, we have
$$\sigma p(a_{(1)})_{(1)} \otimes S(\sigma p(a_{(1)})_{(2)})a_{(2)} \in A \otimes B$$
For any $x \in L$, we have, by (2)
$$\sigma(x)_{(1)} \otimes \sigma(x)_{(2)} \otimes p(\sigma(x)_{(3)})=\sigma(x_{(1)})_{(1)}\otimes  \sigma(x_{(1)})_{(2)} \otimes x_{(2)}$$
and hence for any $a \in A$
 $$\sigma p(a)_{(1)} \otimes \sigma p (a)_{(2)} \otimes p(\sigma p(a)_{(3)})=\sigma p(a_{(1)})_{(1)}\otimes  \sigma p(a_{(1)})_{(2)} \otimes p(a_{(2)})$$
We have
\begin{align*}
({\rm id}_A &\otimes p \otimes {\rm id}_A)   ({\rm id}_A \otimes \Delta) \left(\sigma p(a_{(1)})_{(1)} \otimes S(\sigma p(a_{(1)})_{(2)})a_{(2)}\right) \\
&  =\sigma p(a_{(1)})_{(1)} \otimes Sp(\sigma p(a_{(1)})_{(3)})p(a_{(2)}) \otimes S(\sigma p(a_{(1)})_{(2)})a_{(3)}\\
&= \sigma p(a_{(1)})_{(1)} \otimes Sp(a_{(2)})p(a_{(3)}) \otimes S(\sigma p(a_{(1)})_{(2)})a_{(4)}\\   
&=       \sigma p(a_{(1)})_{(1)} \otimes 1 \otimes 
S(\sigma p(a_{(1)})_{(2)})a_{(2)}                                                                                         
 \end{align*}
and this proves our claim.

We thus get for any $B$-comodule $W$, a linear map
\begin{align*}
\iota : W \boxtimes A & \longrightarrow (W \boxtimes A) \boxtimes B \\
w \otimes a & \longmapsto w \otimes \sigma p(a_{(1)})_{(1)} \otimes S(\sigma p(a_{(1)})_{(2)})a_{(2)}
\end{align*}
that we claim to be a morphism of Yetter-Drinfeld modules over $B$. That $\iota$ is a left $B$-module map is easily checked. Denoting by $\beta$ the $B$-coaction on $(W \boxtimes A) \boxtimes B$, we have
\begin{align*}
  \beta\iota(w \otimes a)  = w_{(0)} &\otimes \sigma p(a_{(1)})_{(2)} \otimes S(\sigma p(a_{(1)})_{(5)})a_{(3)} \otimes \\
& S\left(S(\sigma p(a_{(1)})_{(6)}a_{(2)}\right)S(\sigma p(a_{(1)})_{(1)})w_{(1)}\sigma p(a_{(1)})_{(3)} S(\sigma p(a_{(1)})_{(4)})a_{(4)} \\
 =  w_{(0)} \otimes & \sigma p(a_{(1)})_{(2)} \otimes S(\sigma p(a_{(1)})_{(3)})a_{(3)} \otimes 
S\left(S(\sigma p(a_{(1)})_{(4)})a_{(2)}\right)S(\sigma p(a_{(1)})_{(1)})w_{(1)}a_{(4)} \\
=  w_{(0)} \otimes & \sigma p(a_{(1)})_{(2)} \otimes S(\sigma p(a_{(1)})_{(3)})a_{(3)} \otimes S(a_{(2)})
S\left(\sigma p(a_{(1)})_1S(\sigma p(a_{(1)})_{(4)})\right)w_{(1)}a_{(4)}
\end{align*}
By (3), for $x  \in L$, we have
$$ \sigma(x)_{(2)} \otimes \sigma(x)_{(1)} S(\sigma(x)_{(3)}) =  \sigma(x)\otimes 1_B$$
and hence
$$\sigma(x)_{(2)} \otimes S(\sigma(x)_{(3)})\otimes \sigma(x)_{(1)} S(\sigma(x)_{(4)}) =  \sigma(x)_{(1)}\otimes S(\sigma(x)_{(2)}) \otimes 1_B$$
Thus
\begin{align*}
 \beta\iota(w \otimes a) &= w_{(0)} \otimes \sigma p(a_{(1)})_{(1)} \otimes S(\sigma p(a_{(1)})_{(2)})a_{(3)} \otimes S(a_{(2)})
w_{(1)}a_{(4)} \\
\end{align*}
Now let $\gamma$ be the $B$-coaction on $ W \boxtimes A$. We have
\begin{align*}(\iota\otimes {\rm id}_B)\gamma(w \otimes a)&=\iota\otimes {\rm id}_B(w_{(0)} \otimes a_{(2)} \otimes S(a_{(1)})w_{(1)}a_{(3)})\\
& = w_{(0)} \otimes \sigma p(a_{(2)})_{(1)} \otimes S(\sigma p(a_{(2)})_{(2)})a_{(3)} \otimes S(a_{(1)})w_{(1)}a_{(4)} \\
& = w_{(0)} \otimes \sigma p(a_{(1)})_{(1)} \otimes S(\sigma p(a_{(1)})_{(2)})a_{(3)} \otimes S(a_{(2)})
w_{(1)}a_{(4)}= \beta\iota(w \otimes a)
\end{align*}
where we have used the cocentrality of $p$. It follows that $\iota$ is $B$-colinear, and hence is a morphism of Yetter-Drinfeld modules over $B$. Consider now
\begin{align*}\mu :  (W \boxtimes A) \boxtimes B &\longrightarrow  W \boxtimes A \\
 w \otimes a \otimes b &\longmapsto w \otimes ab
\end{align*}
It is straightforward to check that $\mu$ is a morphism of Yetter-Drinfeld modules over $B$, with $\mu \iota = {\rm id}_{W \boxtimes A}$ and hence we conclude from Proposition \ref{caracrelproj} that  $W \boxtimes A$ is a relative projective Yetter-Drinfeld module over $B$.

For the last assertion, note that $L$ and $A$ both admit right $B^{\rm cop}\otimes L$-comodule structures given by
\begin{align*}
 L &\longrightarrow L \otimes (B^{\rm cop}\otimes L), \quad \quad A \longrightarrow A \otimes (B^{\rm cop}\otimes L) \\
x & \longmapsto x_{(1)}\otimes 1 \otimes x_{(2)}, \quad \quad \ a \longmapsto a_{(2)} \otimes a_{(1)} S(a_{(3)}) \otimes p(a_{(4)})
\end{align*}
and that $p$ is $B^{\rm cop}\otimes L$-colinear. If $A$ is cosemisimple then so is $B$ and so is $L$ (since $p$ is cocentral), hence $B^{\rm cop}\otimes L$ is cosemisimple. Thus there exists a $B^{\rm cop}\otimes L$-colinear section to $p$,  which satisfies our 3 conditions.
\end{proof}

There are also situations where the Hopf algebra in the proposition is not cosemisimple and the map $\sigma$ still exists, see Section 6.

\section{Gerstenhaber-Schack cohomology.}

\subsection{Generalities.}\label{subsec:gen} Let $A$ be a Hopf algebra and let $V$ be a Yetter-Drinfeld module over $A$.
The Gerstenhaber-Schack cohomology of $A$ with coefficients in $V$, that we denote $H_{\rm GS}^*(A,V)$, was introduced in \cite{gs1,gs2} by using an explicit bicomplex. In fact Gerstenhaber-Schack used Hopf bimodules instead of Yetter-Drinfeld modules to define their cohomology, but in view of the equivalence between Hopf bimodules and Yetter-Drinfeld modules, we shall work with the simpler framework of Yetter-Drinfeld modules (a Yetter-Drinfeld version of the Gerstenhaber-Schack bicomplex is provided in \cite{ps}).
A special instance of Gerstenhaber-Schack cohomology is bialgebra cohomology, given by 
$H_b^*(A)=H_{\rm GS}^*(A,\mathbb C)$.

As an example, we have by \cite{pw}, $H_b^*(\C\Gamma) \simeq H^*(\C\Gamma, \C)$ for any discrete group $\Gamma$.
The bialgebra cohomology of $\mathcal O(G)$ for a connected reductive algebraic group $G$ is also described in \cite{pw}, Theorem 9.2, and some finite-dimensional examples are computed in \cite{tai07}. Applications to deformations of pointed Hopf algebras are given in \cite{mw}.

A key result, due to Taillefer \cite{tai04,tai04b}, shows that Gerstenhaber-Schack cohomology is in fact an $\ext$-functor:
$$H_{\rm GS}^*(A,V)\simeq \ext^*_{\yd_A^A}(\mathbb C, V)$$
We will use this description as a definition (we will recall and use the definition based on  a bicomplex in the proof of the forthcoming Proposition \ref{compgs}).

\begin{definition}
 The Gerstenhaber-Schack cohomological dimension of a Hopf  algebra $A$ is
defined to be 
$${\rm cd}_{\rm GS}(A)= {\rm sup}\{n : \ H_{\rm GS}^n(A, V) \not=0 \ {\rm for} \ {\rm some} \ V \in \yd_A^A\}\in \mathbb N \cup \{\infty\}$$   
\end{definition}

If $A$ and $B$ are Hopf algebras having equivalent tensor categories of comodules, then the given tensor equivalence
$F : \mathcal M^A \simeq^\otimes \mathcal M^B$ induces a tensor equivalence $\widehat{F} : \yd_A^A \simeq^\otimes \yd_B^B$ (see e.g. \cite{bicogro,bic}, this is easy to see thanks to the description of the category of Yetter-Drinfeld modules as the weak center of the category of comodules, see \cite{scsurv}). 
Hence we get, for any Yetter-Drinfeld module $V$ over $A$, an isomorphism
$$H_{\rm GS}^*(A,V) \simeq H_{\rm GS}^*(B, \widehat{F}(V))$$  
and moreover ${\rm cd}_{\rm GS}(A)={\rm cd}_{\rm GS}(B)$. These properties are what we call the monoidal invariance of Gerstenhaber-Schack cohomology.

\subsection{Complexes to compute Gerstenhaber-Schack cohomology.}
We now discuss the description of complexes that compute Gerstenhaber-Schack cohomology in particular cases.

Recall that a Hopf algebra $A$ is said to be co-Frobenius if there exists a non-zero $A$-colinear map $A \rightarrow \mathbb C$. By \cite{lin}, $A$ is co-Frobenius if and only if the category $\mathcal M^A$ of right comodules has enough projectives. Finite-dimensional Hopf algebras are co-Frobenius, as well as cosemisimple Hopf algebras.
See \cite{ac,ada} for more examples.

\begin{proposition}\label{resorelative}
 Let $A$ be a co-Frobenius Hopf algebra and let
$$\mathbf P_. = \cdots P_{n+1} \rightarrow P_n \rightarrow \cdots \rightarrow P_1 \rightarrow P_0\rightarrow 0$$
be a resolution of $\mathbb C$ by  projective objects of $\yd_A^A$. We have, for any Yetter-Drinfeld module $V$ over $A$, an isomorphism
$$H_{\rm GS}^*(A,V) \simeq H^*(\Hom_{\yd_A^A}(\mathbf P_.,V))$$
and we have 
$${\rm cd}_{\rm GS}(A) = {\rm inf}\{n: \ \mathbb C\ {\rm admits} \ {\rm a} \ {\rm projective} \  {\rm resolution} \ {\rm of} \ {\rm length} \ n \ {\rm in} \ \yd_A^A\}$$
\end{proposition}

\begin{proof}
We know,  since $A$ is co-Frobenius, that $\yd_A^A$ has enough projective objects (Corollary 3.4 in \cite{bic}). Thus the description of $H_{\rm GS}^*(A,-)$ as an Ext functor  \cite{tai04} yields that
if $\mathbf P_.$ is a a resolution of $\mathbb C$ by projective objects of $\yd_A^A$, we have
$$H_{\rm GS}^*(A,V) \simeq H^*(\Hom_{\yd_A^A}(\mathbf P_.,V))$$
for any Yetter-Drinfeld module $V$. 
The proof of the last statement is similar to the one for group cohomology, see \cite[Chapter VIII, Lemma 2.1]{br}.
\end{proof}

Recall \cite{bic} that for any $n \in \mathbb N$, the Yetter-Drinfeld module
$A^{\boxtimes n}$ is defined as follows:
$$A^{\boxtimes 0}=\C, \ A^{\boxtimes 1} = \C \boxtimes A=A_{\rm coad}, \ A^{\boxtimes 2} = A^{\boxtimes 1} \boxtimes
A, \ \ldots, A^{\boxtimes (n+1)} = A^{\boxtimes n} \boxtimes A, \ldots$$ 
After the obvious vector space identification of 
$A^{\boxtimes n}$ with $A^{\otimes n}$, the right $A$-module structure of $A^{\boxtimes n}$
is given by right multiplication and its comodule structure is given by
\begin{align*}
 {\rm ad}_r^{(n)} : A^{\boxtimes n}&\longrightarrow A^{\boxtimes n} \otimes A \\
 a_1 \otimes \cdots \otimes a_n &\longmapsto a_{1(2)} \otimes \cdots \otimes a_{n(2)}
 \otimes S(a_{1(1)} \cdots a_{n(1)}) a_{1(3)} \cdots a_{n(3)} 
\end{align*}

\begin{proposition}\label{compgs}
Let $A$ be a  Hopf algebra and let $V$ be a Yetter-Drinfeld module over $A$. 
Assume that one of the following conditions holds.
\begin{enumerate}
 \item $A$ is cosemisimple.
\item $V$ is relative injective.
\end{enumerate}
Then the Gerstenhaber-Schack cohomology $H^*_{\rm GS}(A,V)$ is the cohomology of the complex
 $$ 0 \rightarrow {\rm Hom}^A(\C,V) \overset{\partial}\longrightarrow \Hom^A(A^{\boxtimes 1},V) \overset{\partial}\rightarrow
 \cdots \overset{\partial}\rightarrow\Hom^A(A^{\boxtimes n}, V) \overset{\partial}\longrightarrow \Hom^A(A^{\boxtimes n+1}, V) \overset{\partial}\longrightarrow \cdots$$
where the differential $\partial : \Hom^A(A^{\boxtimes n}, V) \longrightarrow \Hom^A(A^{\boxtimes n+1},V)$
is given by 
\begin{align*}
\partial(f)(a_1 \otimes \cdots \otimes a_{n+1}) = &\varepsilon(a_1)
f(a_2 \otimes \cdots \otimes a_{n+1}) + \sum_{i=1}^{n}(-1)^i f(a_1 \otimes \cdots \otimes a_i a_{i+1} \otimes \cdots \otimes a_{n+1}) \\
&+ (-1)^{n+1} f(a_1 \otimes \cdots \otimes a_{n}) \cdot a_{n+1}
\end{align*}
\end{proposition}

\begin{proof}
By \cite{bic}, Proposition 3.6, the standard resolution of $\mathbb C_\varepsilon$ yields in a fact resolution of $\C$ by free Yetter-Drinfeld modules in the category $\yd_A^A$
$$\cdots \longrightarrow  A^{\boxtimes n+1} \longrightarrow A^{\boxtimes n} \longrightarrow \cdots 
\longrightarrow A^{\boxtimes 2}\longrightarrow A^{\boxtimes 1} \longrightarrow 0 $$
where each differential is given by
\begin{align*}
 A^{\boxtimes n+1} &\longrightarrow A^{\boxtimes n} \\
 a_1 \otimes \cdots \otimes a_{n+1} &\longmapsto \varepsilon(a_1) a_2 \otimes \cdots \otimes a_{n+1} + 
 \sum_{i=1}^n(-1)^i a_1 \otimes \cdots \otimes a_i a_{i+1} \otimes \cdots \otimes a_{n+1}
\end{align*}
If $A$ is cosemimple, then free Yetter-Drinfeld modules are projective, and we get, after standard identification using the fact that the free functor is left adjoint, the result by 
Proposition \ref{resorelative}. 

To prove the result if the second condition holds, we recall the definition of Gerstenhaber-Schack cohomology using a bicomplex \cite{shst}. 
Let $V,W$ be objects in $\yd_A^A$, let $P_{\bullet} \rightarrow V\rightarrow 0$ be a relative projective resolution of $V$ (this means that the objects $P_q$, $q\geq 0$, are relative projective and the the sequence $P_{\bullet} \rightarrow V\rightarrow 0$ splits as a sequence of comodules), and let $0\rightarrow W \rightarrow I^\bullet$ be a relative injective resolution of $W$ (this means that the objects $I^p$, $p\geq 0$, are relative injective and the the sequence $0\rightarrow W \rightarrow I^{\bullet}$ splits as a sequence of modules).
We then can form, in a standard way, the bicomplex $C^{\bullet, \bullet}(V,W)={\rm Hom}_{\yd_A^A}(P_\bullet, I^\bullet)$. The uniqueness, up to homotopy, of the previous resolutions (\cite{shst}, chapter 10) shows that the cohomology of the bicomplex $C^{\bullet, \bullet}(V,W)={\rm Hom}_{\yd_A^A}(P_\bullet, I^\bullet)$ is independent of the choice of these resolutions, and  is the Gertenhaber-Schack cohomology of the Yetter-Drinfeld modules $V$ and $W$ (see \cite{tai04b,tai04} as well). When $V=\C$, we get the Gerstenhaber Schack-cohomology $H^*_{\rm GS}(A,W)$ as defined in Subsection \ref{subsec:gen}, by \cite{tai04}.

Assuming that $W$ is relative injective, we can use the relative injective resolution $$0 \rightarrow W \rightarrow W \rightarrow 0 \rightarrow \cdots \rightarrow 0 \cdots$$
together with the standard resolution of the trivial object $\C$ as above (which is indeed a relative projective resolution of $\mathbb C$), and we get a bicomplex with only one non-zero column, which is, again using the fact that the free functor is left adjoint, easily identified with the complex in the statement of the proposition.
\end{proof}

\begin{rem}
When $V = \mathbb C$ is the trivial Yetter-Drinfeld module, the complex in Theorem \ref{compgs} is the same as the one defined in \cite{ger} in the study of additive deformations of Hopf algebras, which are of interest in quantum probability. This complex is also the complex that defines the so-called Davydov-Yetter cohomology of the tensor category $\mathcal M^A$ (\cite{dav,ye}, see \cite{egno}, Chapter 7).
\end{rem}

\begin{rem}\label{nonembedd}
Let $V$ be a Yetter-Drinfeld module over $A$. The complex in Proposition \ref{compgs} is a subcomplex of the complex that computes the Hochshild cohomology $H^*(A, {_\varepsilon \!V})$, where the left $A$-module structure on ${_\varepsilon \!V}$ is the one induced by the counit and the right module structure is the original one. We thus have a linear map $$H^*_{\rm GS}(A,V) \rightarrow H^*(A, {_\varepsilon \!V})\simeq {\rm Ext}_A(\mathbb C_\varepsilon, V)$$
which is not injective in general. Indeed for $q \in \mathbb C^*$ generic ($q=\pm 1$ or not a root of unity), we have $H^3_{\rm GS}(\mathcal O(\SL_q(2)), \mathbb C) \simeq \mathbb C$ (see \cite{bic}), while  $H^3(\mathcal O(\SL_q(2)),  {_\varepsilon \! \mathbb C}_\varepsilon)=0$ if $q^2\not=1$ (see e.g. \cite{hk}). In Subsection \ref{cosemi} we provide some conditions that ensure that the above map is injective.
\end{rem}

\subsection{Relation with Hochschild cohomology}
We are  ready to show that the Gerstenhaber-Schack cohomology of a Hopf algebra determines its Hochschild cohomology.

\begin{theorem}\label{gsh}
 Let $A$ be a  Hopf algebra and let $M$ be an $A$-bimodule. Endow $M \otimes A$ with a Yetter-Drinfeld module structure defined by
$$m \otimes a \mapsto m \otimes a_{(1)} \otimes a_{(2)}, \ (m \otimes a)\leftarrow b= S(b_{(2)}).m.b_{(3)}\otimes S(b_{(1)})ab_{(4)}, \ a,b \in A, \ m \in M$$  
and denote by $M'\#A$ the resulting Yetter-Drinfeld module. Then we have an isomorphism
$$H^*(A,M) \simeq H_{\rm GS}^*(A, M'\#A)$$
In particular we have ${\rm cd}(A) \leq {\rm cd}_{\rm GS}(A)$. 
\end{theorem}

\begin{proof}
The Yetter-Drinfeld module $M'\#A$ is the co-free Yetter-Drinfeld associated to  the right $A$-module $M'$ of Section 2. It is thus a relative injective Yetter-Drinfeld module (Proposition \ref{caracrelinj}), and we can use the complex of Proposition \ref{compgs} to compute its Gerstenhaber-Schack cohomology.

Recall that since $H^*(A,M)\simeq \ext_A^*(\mathbb C_\varepsilon, M')$ (Section 2), the complex to compute 
$H^*(A,M)$ is 
$$ 0 \longrightarrow {\rm Hom}(\C,M') \overset{\partial}\longrightarrow \Hom(A,M') \overset{\partial}\longrightarrow
\cdots \overset{\partial}\longrightarrow\Hom(A^{\otimes n}, M') \overset{\partial}\longrightarrow \Hom(A^{\otimes n+1}, M') \overset{\partial}\longrightarrow \cdots$$
where the differential $\partial : \Hom(A^{\otimes n}, M') \longrightarrow \Hom(A^{\otimes n+1}, M')$
is given by 
\begin{align*}
\partial(f)(a_1 \otimes \cdots \otimes a_{n+1}) = &\varepsilon(a_1)
f(a_2 \otimes \cdots \otimes a_{n+1}) + \sum_{i=1}^{n}(-1)^i f(a_1 \otimes \cdots \otimes a_i a_{i+1} \otimes \cdots \otimes a_{n+1}) \\
&+ (-1)^{n+1} S(a_{n+1(1)})\cdot f(a_1 \otimes \cdots \otimes a_{n}) \cdot a_{{n+1}(2)}
\end{align*}
For all $n \geq 0$,  we have  linear isomorphisms
\begin{align*}
 \Hom^A(A^{\boxtimes n}, M'\#A) & \longrightarrow {\rm Hom}(A^{\otimes n}, M') \\
f & \longmapsto ({\rm id}_M \otimes \varepsilon) f 
\end{align*}
For $f \in  \Hom(A^{\boxtimes n}, M'\#A)$ and $a_1, \ldots , a_{n} \in A$, with $f(a_1\otimes \cdots \otimes a_{n})=
\sum_i m_i\otimes b_i$, we have
\begin{align*}
 {\rm id}_M \otimes &\varepsilon(f(a_1 \otimes \cdots \otimes a_{n})\leftarrow a_{n+1})\\
& =   {\rm id}_M \otimes \varepsilon\left(\sum_iS(a_{n+1(2)}).m_i.a_{n+1(3)} \otimes S(a_{n+1(1)})b_ia_{n+1(4)}\right)\\
&= \sum_i\varepsilon(b_i) S(a_{n+1(1)}).m_i.a_{n+1(2)} \\
&= S(a_{n+1(1)}).\left(({\rm id}_M \otimes \varepsilon)(f(a_1 \otimes \cdots \otimes a_{n+1})\right).a_{n+1(2)}
\end{align*}
From this computation it follows easily that the previous isomorphisms commute with the differentials (as already said, the one for Gerstenhaber-Schack cohomology being given by the complex of Proposition \ref{compgs}), and hence the complexes that define both cohomologies are isomorphic.
\end{proof}

We get the results announced in the introduction, providing a partial answer to Question \ref{ques}.

\begin{corollary}
Let  $A$ and $B$ be  Hopf algebras  such that 
there exists an equivalence of linear tensor categories
 $\mathcal M^A \simeq^{\otimes} \mathcal M^B$. Then there exist two functors 
$$F_1 : {\!_A\mathcal M}_A \rightarrow  \yd_B^B\ {\rm and} \ F_2 : {\!_B\mathcal M}_B \rightarrow \yd_A^A$$
such that for any $A$-bimodule $M$ and any  $B$-bimodule $N$, we have
$$H^*(A,M)\simeq H_{\rm GS}^*(B,F_1(M)) \ {\rm and} \ H^*(B,N)\simeq H_{\rm GS}^*(A,F_2(N))$$
In particular we have
$\max({\rm cd}(A),{\rm cd}(B))\leq {\rm cd}_{\rm GS}(A)={\rm cd}_{\rm GS}(B)$.
\end{corollary}

\begin{proof}
The construction in the previous theorem clearly yields a functor ${\!_A\mathcal M}_A \rightarrow  \yd_A^A$, that we compose with the functor $\yd_A^A \rightarrow \yd_B^B$ from the discussion at the end of subsection 5.1, to get the announced functor $F_1$, and similarly the functor $F_2$. The last claim follows immediately.
\end{proof}

\begin{rem}\label{rem:lr}
Recall that Question \ref{ques2}, motivated by Theorem \ref{gsh}, asks if ${\rm cd}(A)= {\rm cd}_{\rm GS}(A)$ for any  Hopf algebra $A$. 
Question \ref{ques2} has indeed a positive answer in the finite-dimensional case: if $A$ is semisimple, then it is cosemisimple by the Larson-Radford theorem \cite{lr}, and hence $\yd_A^A$ is semisimple (since
the Drinfeld double $D(A)$ is then semisimple, see \cite{rad93}), so we have ${\rm cd}(A)=0={\rm cd}_{\rm GS}(A)$.
If $A$ is not semisimple, then ${\rm cd}(A)=\infty={\rm cd}_{\rm GS}(A)$.
It thus follows that a positive answer to Question \ref{ques2} would provide a natural infinite-dimensional generalization to the above mentioned Larson-Radford theorem.

The characteristic zero assumption is indeed necessary: if $A$ is a finite-dimensional semisimple non cosemisimple Hopf algebra, the base field being then necessarily of characteristic $>0$ \cite{lr}, then ${\rm cd}(A)=0<{\rm cd}_{\rm GS}(A)=\infty$.

See the next subsection for some partial results in the cosemisimple case.
\end{rem}

\subsection{Cosemisimple Hopf algebras}\label{cosemi} We now provide some more precise partial answers to Questions \ref{ques} and \ref{ques2} when the Hopf algebra is cosemisimple and of Kac type (recall that this means that $S^2={\rm id}$).

\begin{proposition}\label{embeddgs}
 Let $A$ be a cosemisimple Hopf algebra of Kac type, and let $V$ be a Yetter-Drinfeld module over $A$. Then the natural linear map 
$$H^*_{\rm GS}(A,V) \rightarrow H^*(A, {_\varepsilon \!V})$$
arising from Proposition \ref{compgs} is injective.
\end{proposition}

\begin{proof}
 Let $h$ be the Haar integral on $A$. Recall that for any $A$-comodules $V$ and $W$, we have a surjective averaging operator 
\begin{align*}
 M : {\rm Hom}(V,W) & \longrightarrow {\rm Hom}^A(V,W) \\
f & \longmapsto M(f), \ M(f)(v)= h\left(f(v_{(0)})_{(1)}S(v_{(1)})\right)f(v_{(0)})_{(0)}
\end{align*}
with $f \in {\rm Hom}^A(V,W)$ if and only if $M(f)=f$.
Now let $V$ be our given Yetter-Drinfeld module, and let $f \in {\rm Hom}(A^{\otimes n}, V)$. We thus have
$M(f) \in  {\rm Hom}^A(A^{\boxtimes n}, V)$, with 
\begin{align*}M(f)(a_1 \otimes \cdots \otimes a_n) = &h\left(f(a_{1(2)} \otimes \cdots \otimes a_{n(2)})_{(1)}S(a_{1(3)} \cdots a_{n(3)})S^2(a_{1(1)} \cdots a_{n(1)})\right) \\ & f(a_{1(2)} \otimes \cdots \otimes a_{n(2)})_{(0)}\end{align*}
It is a tedious but straightforward verification to check that, under our assumption, we have $\partial(M(f))= M(\partial(f))$. 
To convince the reader, we present the verification at $n=2$. Let $f \in  {\rm Hom}(A^{\otimes 2}, V)$. We have 
\begin{align*}
\partial(M(f)) (a \otimes b \otimes c)  = & \varepsilon(a)h\left( f(b_{(2)}\otimes c_{(2)})_{(1)}S(b_{(3)} c_{(3)})S^2(b_{(1)} c_{(1)}) \right) f(b_{(2)}\otimes c_{(2)})_{(0)} \\
&  - h\left(f(a_{(2)}b_{(2)}\otimes c_{(2)})_{(1)} S(a_{(3)} b_{(3)} c_{(3)})S^2(a_{(1)} b_{(1)} c_{(1)}) \right) f(a_{(2)}b_{(2)}\otimes c_{(2)})_{(0)}\\
& + h\left(f(a_{(2)}\otimes b_{(2)} c_{(2)})_{(1)} S(a_{(3)} b_{(3)} c_{(3)})S^2(a_{(1)} b_{(1)} c_{(1)}) \right) f(a_{(2)}\otimes b_{(2)}c_{(2)})_{(0)}\\
& - h\left(f(a_{(2)}\otimes b_{(2)})_{(1)} S(a_{(3)} b_{(3)})S^2(a_{(1)} b_{(1)}) \right) f(a_{(2)}\otimes b_{(2)})_{(0)}\cdot c
\end{align*}
On the other hand we have
\begin{align*}
 M(\partial(f))(a \otimes b & \otimes c) =  h\left(\partial(f)(a_{(2)} \otimes b_{(2)}\otimes c_{(2)})_{(1)}S(a_{(3)}b_{(3)} c_{(3)})S^2(a_{(1)}b_{(1)} c_{(1)}) \right) \\ & \quad \quad \quad  \quad \partial(f)(a_{(2)}\otimes b_{(2)}\otimes c_{(2)})_{(0)} \\
 & = h\left( \varepsilon(a_{(2)})f(b_{(2)}\otimes c_{(2)})_{(1)}S(a_{(3)}b_{(3)} c_{(3)})S^2(a_{(1)}b_{(1)} c_{(1)}) \right) f(b_{(2)}\otimes c_{(2)})_{(0)} \\
&  - h\left(f(a_{(2)}b_{(2)}\otimes c_{(2)})_{(1)} S(a_{(3)} b_{(3)} c_{(3)})S^2(a_{(1)} b_{(1)} c_{(1)}) \right) f(a_{(2)}b_{(2)}\otimes c_{(2)})_{(0)}\\
& + h\left(f(a_{(2)}\otimes b_{(2)} c_{(2)})_{(1)} S(a_{(3)} b_{(3)} c_{(3)})S^2(a_{(1)} b_{(1)} c_{(1)}) \right) f(a_{(2)}\otimes b_{(2)}c_{(2)})_{(0)}\\
& - h\left((f(a_{(2)}\otimes b_{(2)})\cdot c_{(2)})_{(1)} S(a_{(3)} b_{(3)}c_{(3)})S^2(a_{(1)} b_{(1)}c_{(1)}) \right) ((f(a_{(2)}\otimes b_{(2)})\cdot c_{(2)})_{(0)}
\end{align*}
Using the Yetter-Drinfeld condition, the last expression equals
$$h\left(S(c_{(2)})f(a_{(2)}\otimes b_{(2)})_{(1)} c_{(4)} S(a_{(3)} b_{(3)}c_{(5)})S^2(a_{(1)} b_{(1)}c_{(1)}) \right) (f(a_{(2)}\otimes b_{(2)})_{(0)}\cdot c_{(3)}$$
The fact that $S^2={\rm id}$ and that the Haar integral is a trace (since $S^2={\rm id}$) then shows that this last expression equals the last one in the computation of $\partial(M(f)) (a \otimes b \otimes c)$, and shows that indeed
$\partial(M(f))= M(\partial(f))$.

Now let $f \in {\rm Hom}^A(A^{\boxtimes n}, V)$ be such that $f=\partial(\mu)$ for some $\mu \in {\rm Hom}(A^{\otimes n-1}, V)$. Then $M(f)=M(\partial(\mu))=\partial (M(\mu)))$, with  $M(\mu) \in {\rm Hom}^A(A^{\boxtimes n-1}, V)$, and hence $f=0$ in $H^n_{\rm GS}(A,V)$: our claim is proved. 
\end{proof}

We thus get the following partial answers to Questions \ref{ques2} and \ref{ques}.

\begin{corollary}\label{cor:partansques2}
 Let $A$ be cosemisimple Hopf algebra of Kac type. Then ${\rm cd}(A)={\rm cd}_{\rm GS}(A)$.
\end{corollary}

\begin{proof}
 We have ${\rm cd}(A)\leq {\rm cd}_{\rm GS}(A)$ by Theorem \ref{gsh}, and the previous proposition ensures that  ${\rm cd}_{\rm GS}(A)\leq {\rm cd}(A)$.
\end{proof}

\begin{corollary}\label{cor:partansques}
 Let  $A$ and $B$ be cosemisimple Hopf algebras  such that there exists an equivalence of linear tensor categories
 $\mathcal M^A \simeq^{\otimes} \mathcal M^B$. If $A$ is of Kac type, then we have ${\rm cd}(A) \geq {\rm cd}(B)$, and if $A$ and $B$ both are of Kac type, then ${\rm cd}(A) = {\rm cd}(B)$.
\end{corollary}

\begin{proof}
 We have, combining  Theorem \ref{gsh} and the previous corollary,
$$ {\rm cd}(A)={\rm cd}_{\rm GS}(A)={\rm cd}_{\rm GS}(B)\geq {\rm cd}(B)$$
with ${\rm cd}(B)={\rm cd}_{\rm GS}(B)$ if $B$ is of Kac type as well.
\end{proof}

See the next section for examples that are not of Kac type.

\section{Application to quantum symmetry algebras}

In this section we provide applications of the previous considerations to  quantum symmetry
algebras.

\subsection{The universal Hopf algebra of a non-degenerate bilinear form and its adjoint subalgebra}

 Let $E \in \GL_n(\C)$. Recall that the algebra $\mathcal B(E)$ \cite{dvl} is presented by generators 
$(u_{ij})_{1 \leq i,j\leq n}$ and relations
$$E^{-\!1} u^t E u = I_n = u E^{-\!1} u^t E,$$
where $u$ is the matrix $(u_{ij})_{1 \leq i,j \leq n}$. It has a Hopf algebra structure defined by
$$\Delta(u_{ij})
= \sum_{k=1}^n u_{ik} \otimes u_{kj}, \ 
\varepsilon(u_{ij}) = \delta_{ij}, \ 
S(u) = E^{-1}u^t E$$
The Hopf algebra $\mathcal B(E)$ represents the quantum symmetry group of the bilinear form associated to the matrix $E$. It can also be constructed as a quotient of the FRT bialgebra associated to Yang-Baxter operators constructed by Gurevich \cite{gu}.
For the matrix $$E_q = \begin{pmatrix} 0 & 1 \\ -q^{-1} & 0\end{pmatrix}$$ we have $\B(E_q) = \mathcal O(\SL_q(2))$, and thus the Hopf algebras $\B(E)$ are natural generalizations of $\mathcal O(\SL_q(2))$.
It is shown in \cite{bi1} that
for $q \in \C^*$ satisfying  ${\rm tr}(E^{-1}E^t)= -q-q^{-1}$, the tensor categories of comodules over $\B(E)$ and $\mathcal O(\SL_q(2))$ are equivalent. Thus $\mathcal B(E)$ is cosemisimple if and only if the corresponding $q$ is not a root of unity or $q=\pm1$.

It was proved in \cite{bic} that if $n\geq 2$, then ${\rm cd}(\B(E))=3$ (Theorem 6.1 and Proposition 6.4 in \cite{bic}, see e.g. \cite{hk} for the case $E=E_q$ and \cite{cht} for the case $E=I_n$), and the bialgebra cohomology of $\mathcal B(E)$ was computed there in the cosemisimple case.

As a preliminary step towards the study of quantum symmetry algebras of semisimple algebras, we now study the adjoint subalgebra $\mathcal B_+(E)$ of $\mathcal B(E)$.

The algebra $\mathcal B_+(E)$ is, by definition, the subalgebra of $\mathcal B(E)$ generated by the elements $u_{ij}u_{kl}$, $1 \leq i,j,k,l\leq n$. It is easily seen to be a Hopf subalgebra. Also it is easily seen that 
$\mathcal B_+(E)=\mathcal B(E)^{{\rm co}\C \mathbb Z_2}$, where $p$ is the cocentral Hopf algebra map $\mathcal B(E) \rightarrow \C \mathbb Z_2$, $u_{ij} \mapsto \delta_{ij}g$, where $g$ stands for the generator of $\mathbb Z_2$, the cyclic group of order $2$. The Hopf algebra $\mathcal B_+(E)$ is cosemisimple if and only if $\mathcal B(E)$ is.

\begin{lemma}
Assume that ${\rm tr}(E^{-1}E^t)\not=0$. Then there exists a linear map $\sigma : \mathbb C \mathbb Z_2 \rightarrow \mathcal B(E)$ satisfying the conditions of Proposition \ref{sigma}.
\end{lemma}

\begin{proof}
  Consider the matrix $F=E(E^{t})^{-1}=(\alpha_{ij})$. We have ${\rm tr}(F)={\rm tr}(E^{-1}E^t)=t\not=0$. Consider the element $x = t^{-1} \sum_{ij}\alpha_{ij}u_{ij} \in \mathcal B(E)$ and let $\sigma : \mathbb C \mathbb Z_2 \rightarrow \mathcal B(E)$ be the unique linear map such that $\sigma(1)=1$ and $\sigma(g)=x$.
It is straightforward to check that $\sigma$ indeed satisfies the conditions of Proposition \ref{sigma}.
\end{proof}

\begin{theorem}\label{cdb+e}
 Let  $E \in {\rm GL}_n(\mathbb C)$ with $n\geq 2$. Then we have ${\rm cd}(\mathcal B_+(E))= 3\leq {\rm cd}_{\rm GS}(\mathcal B_+(E))$, and if moreover $\mathcal B_+(E)$ is cosemisimple, then ${\rm cd}_{\rm GS}(\mathcal B_+(E))=3$.
\end{theorem}

\begin{proof}
We have, by Proposition \ref{refoexact}, a strict exact sequence of Hopf algebras
$$\C  \to \mathcal B_+(E) \to \mathcal B(E) \to \C \mathbb Z_2 \to \C$$
so it follows from Proposition \ref{subadd} that ${\rm cd}(\mathcal B_+(E))= {\rm cd}(\mathcal B(E))=3$. By Theorem \ref{gsh} we have  ${\rm cd}_{\rm GS}(\mathcal B_+(E))\geq 3$.

Consider now the exact sequence of free Yetter-Drinfeld modules over $\B(E)$ from \cite{bic}:
 $$0 \to \C\boxtimes \mathcal B(E) \overset{\phi_1}\longrightarrow
  (V_E^* \otimes V_E) \boxtimes \B(E) \overset{\phi_2}\longrightarrow
 (V_E^* \otimes V_E) \boxtimes \B(E) \overset{\phi_3} \longrightarrow \C \boxtimes \B(E) \overset{\varepsilon} \longrightarrow \C \to 0 $$
All the $\mathcal B(E)$-comodules involved in the left terms are in fact comodules over $\mathcal B_+(E)$, so we have, by Proposition \ref{caradj}, an exact sequence of  Yetter-Drinfeld modules over $\mathcal B_+(E)$. Assume now that $\mathcal B_+(E)$ is cosemisimple. The previous lemma ensures that we are in the situation of Proposition \ref{sigma}, so all the terms in the sequence (except the last one of course) are projective Yetter-Drinfeld modules over $\mathcal B_+(E)$. We conclude from Proposition \ref{resorelative} that ${\rm cd}_{\rm GS}(\mathcal B_+(E))\leq 3$, and hence that ${\rm cd}_{\rm GS}(\mathcal B_+(E))=3$.
\end{proof}

To compute the bialgebra cohomology of $\mathcal B_+(E)$ in the cosemisimple case, we need some preliminaries.
We specialize at $E_q = \begin{pmatrix} 0 & 1 \\ -q^{-1} & 0\end{pmatrix}$ and we put $A=\mathcal B(E_q)=  \mathcal O(\SL_q(2))$ (with its standard generators $a$, $b$, $c$, $d$) and $B=\mathcal B_+(E_q)$. In the next lemma we only assume that $q+q^{-1} \not=0$. Recall from Subsection 4.4 that if $W$ is a $B$-comodule, then $W \boxtimes A$ is a Yetter-Drinfeld module over $B$.

\begin{lemma}\label{sigmab+}
 We have, for any $B$-comodule $W$, a vector space isomorphism
\begin{align*}
 {\rm Hom}_{\yd_B^B}(W\boxtimes A, \mathbb C) & \longrightarrow {\rm Hom}^B(W,\mathbb C)\oplus {\rm Hom}^B(W,\C) \\
\psi & \longmapsto (\psi(-\otimes 1),\psi(-\otimes \chi))
\end{align*}
where $\chi = q^{-1}a +qd$.
\end{lemma}

\begin{proof}
 Let $\psi \in{\rm Hom}_{\yd_B^B}(W\boxtimes A, \mathbb C)$.
That both $\psi(-\otimes 1)$ and $\psi(-\otimes \chi)$ are $B$-comodule maps follow from the fact that $1$ and $\chi$ are coinvariant for the co-adjoint action of $A$. We have, for any $w \in W$, using the $B$-linearity
$$\psi(w \otimes b) = \psi(w \otimes b(ad-q^{-1}bc))=\psi(w\otimes bad)=q\psi(w\otimes abd)=0$$
and similarly $\psi(w\otimes c)=0$. We also have
$$\psi(w \otimes d)=\psi(w\otimes d(ad-q^{-1}bc)))=\psi(w \otimes dad)=\psi(w\otimes ad^2)=\psi(w \otimes a)$$ 
These identities, together with the fact that $A=B\oplus B'$, where $B'=XB$ and $X=\{a,b,c,d\}$, show that the map in the statement of the lemma is injective. 

For $(\psi_1,\psi_2) \in  {\rm Hom}^B(W,\mathbb C)\oplus {\rm Hom}^B(W,\C)$, we define a linear map $\psi : W \otimes A \rightarrow \C$ by
$$\psi(w \otimes (y+y')) = \psi_1(w)\varepsilon(y)+ (q+q^{-1})^{-1}\psi_2(w)\varepsilon(y'), \ y \in B, \ y \in B'$$
It is clear that $\psi$ is $A$-linear and a direct verification to check that $\psi$ is a map of $B$-comodules, for the co-action of $W\boxtimes A$. Hence we have $\psi \in{\rm Hom}_{\yd_B^B}(W\boxtimes A, \mathbb C)$, and clearly
$\psi(-\otimes 1)=\psi_1$ and $\psi(-\otimes \chi)=\psi_2$. Therefore our map is surjective, and we are done.
\end{proof}

\begin{theorem}\label{cobib+e}
 Let  $E \in {\rm GL}_n(\mathbb C)$ with $n\geq 2$. If $\mathcal B_+(E)$ is cosemisimple, then  \begin{equation*}
  H^n_{b}(\B_+(E)) \simeq \begin{cases}
                         0 & \text{if $n \not=0,3$} \\
			 \C & \text{if $n=0,3$}
                        \end{cases}
 \end{equation*}
\end{theorem}

\begin{proof}
 The monoidal invariance of bialgebra cohomology enables us to assume that $E=E_q$ as in the previous discussion, of which we keep the notations. We denote by $V$ the fundamental $A$-comodule of dimension $2$, of which we fix a basis $e_1$, $e_2$. We have an exact sequence of Yetter-Drinfeld modules over $A$ (and over $B$)
 $$0 \to \C\boxtimes A \overset{\phi_1}\longrightarrow
  (V^* \otimes V) \boxtimes A \overset{\phi_2}\longrightarrow
 (V^* \otimes V) \boxtimes A \overset{\phi_3} \longrightarrow \C \boxtimes A \overset{\varepsilon} \longrightarrow \C \to 0$$
with for any $x \in A$ (see the proof of Lemma 5.6 in \cite{bic})
 \begin{align*}
 \phi_1(x) = e_1^* \otimes e_1 \otimes &((-q^{-1} +qd)x )+
 e_1^* \otimes e_2 \otimes (-cx) \\ &+ e_2^* \otimes e_1 \otimes (-bx) +
 e_2^* \otimes e_2 \otimes ((-q+q^{-1}a)x) \\
  \phi_2(e_1^* \otimes e_1 \otimes x) &= e_1^* \otimes e_1 \otimes x
  + e_2^* \otimes e_1 \otimes (-qbx) + e_2^* \otimes e_2 \otimes ax \\
\phi_2 (e_1^* \otimes e_2 \otimes x) &=
e_1^* \otimes e_1 \otimes bx + e_1^* \otimes e_2 \otimes (1-q^{-1}a)x \\
\phi_2(e_2^* \otimes e_1 \otimes x) &= e_2^* \otimes e_1 \otimes (1-qd)x
+e_2^* \otimes e_2 \otimes cx \\
\phi_2(e_2^* \otimes e_2 \otimes x) &= e_1^* \otimes e_1 \otimes dx + e_1^* \otimes e_2 \otimes (-q^{-1}cx)
+ e_2^* \otimes e_2 \otimes x \\
\phi_3(e_1^* \otimes e_1 \otimes x) &= (a-1)x, \quad \phi_3(e_1^* \otimes e_2 \otimes x) = bx, \\
\phi_3(e_2^* \otimes e_1 \otimes x) &=cx, \quad \phi_3(e_2^* \otimes e_2 \otimes x) =(d-1)x
 \end{align*} 
and by Lemma \ref{sigmab+}, Proposition \ref{sigma}  and Proposition \ref{resorelative}, the bialgebra cohomology of $B$ is the cohomology of the complex
{\footnotesize
$$0\rightarrow {\rm Hom}_{\yd_B^B}(\C \boxtimes A,\C) \overset{\phi_3^t} \rightarrow {\rm Hom}_{\yd_B^B}(V^* \otimes V) \boxtimes A,\C)  \overset{\phi_2^t} \rightarrow {\rm Hom}_{\yd_B^B} (V^* \otimes V) \boxtimes A,\C) \overset{\phi_1^t} \rightarrow {\rm Hom}_{\yd_B^B}(\C \boxtimes A,\C) \rightarrow 0
$$}
We have, by the previous lemma,
${\rm Hom}_{\yd_B^B}(\C \boxtimes A,\C) \simeq \C^2$,  and
$$ {\rm Hom}_{\yd_B^B} (V^* \otimes V) \boxtimes A,\C)\simeq {\rm Hom}^B (V^* \otimes V,\C) \oplus {\rm Hom}^B (V^* \otimes V,\C)\simeq \C^2$$
Therefore the previous complex is isomorphic to a complex of the form
 $$0 \longrightarrow \C^2\longrightarrow \C^2 \longrightarrow \C^2 \longrightarrow \C^2 \longrightarrow  0$$
The reader will easily write down explicitly this complex and compute its cohomology, yielding the announced result for the bialgebra cohomology of $B$.
\end{proof}

\subsection{Bialgebra cohomology and cohomological dimensions of $A_{\rm aut}(R,\varphi)$}
Let $(R,\varphi)$ be a finite-dimensional measured algebra: this means that $R$ is a finite-dimensional algebra and $\varphi : R \rightarrow \C$ is a linear map (a measure on $R$) such that the associated bilinear map $R \times R \rightarrow \C$, $(x,y) \mapsto \varphi(xy)$ is non-degenerate. Thus a  finite-dimensional measured algebra is a  Frobenius algebra together with a fixed measure.
A coaction of a Hopf algebra $A$ on a finite-dimensional measured algebra $(R,\varphi)$ is an $A$-comodule structure on $R$ making it into an $A$-comodule algebra and such that $\varphi : R \rightarrow \C$ is $A$-colinear.
It is well-known that there exists a universal Hopf algebra coacting on $(R,\varphi)$ (see \cite{wan98} in the compact case with $R$ semisimple and \cite{bi00} in general), that we denote  $A_{\rm aut}(R,\varphi)$ and call the quantum symmetry algebra of $(R,\varphi)$.
The following particular cases are of special interest.

\begin{enumerate}
 \item For $R=\C^n$ and $\varphi=\varphi_n$ the canonical integration map (with $\varphi_n(e_i)=1$ for $e_1, \ldots , e_n$ the canonical basis of $\C^n$), we have  $A_{\rm aut}(\C^n,\varphi_n)=:A_s(n)$, the coordinate algebra on the quantum permutation group \cite{wan98}, presented by generators $x_{ij}$, $1 \leq i,j\leq n$, submitted to the relations
$$\sum_{l=1}^nx_{li}=1= \sum_{l=1}^nx_{il}, \ x_{ik}x_{ij}= \delta_{kj}x_{ij}, \  x_{ki}x_{ji}= \delta_{kj}x_{ji}, \ 1\leq i,j,k\leq n$$
Its Hopf algebra structure is defined by 
$$\Delta(x_{ij})
= \sum_{k=1}^n x_{ik} \otimes x_{kj}, \ 
\varepsilon(x_{ij}) = \delta_{ij}, \ 
S(x_{ij}) = x_{ji}$$
The Hopf algebra $A_s(n)$ is infinite-dimensional if $n\geq 4$ \cite{wan98}.
\item For $R=M_2(\C)$ and $q \in \C^*$, let ${\rm tr}_q : M_2(\C) \rightarrow  \C$ be the $q$-trace, i.e. ${\rm tr}_q(g) =qg_{11}+q^{-1}g_{22}$ for $g=(g_{ij})\in M_2(\C)$. Then we have $A_{\rm aut}(M_2(\C),{\rm tr}_q) \simeq \mathcal O({\rm PSL}_q(2))$, the latter algebra being $\B_+(E_q)$ in the notation of the previous subsection (it is often denoted $\mathcal O({\rm SO}_{q^{1/2}}(3))$, see e.g. \cite{ks}). The above isomorphism $A_{\rm aut}(M_2(\C),{\rm tr}_q) \rightarrow \mathcal O({\rm PSL}_q(2))$ is constructed using the universal property  of $A_{\rm aut}(M_2(\C),{\rm tr}_q)$, and the verification that it is indeed injective is a long and tedious computation, as in \cite{dij}.
\end{enumerate}

Let $(R,\varphi)$ be a finite-dimensional measured algebra. Since $\varphi \circ m$ is non-degenerate, where $m$ is the multiplication of $R$, there exists a linear map $\delta : \C \rightarrow R \otimes R$ such that 
$(R,\varphi \circ m, \delta)$ is a left dual for $R$, i.e.
$$((\varphi \circ m)\otimes {\rm id}_R) \circ ({\rm id}_R \otimes \delta)={\rm id}_R= 
({\rm id}_R \otimes  (\varphi \circ m)) \circ (\delta \otimes {\rm id}_R)$$
 Following \cite{mro}, we put
$$\tilde{\varphi}= \varphi \circ m \circ (m \otimes {\rm id}_R) \circ ({\rm id}_R \otimes  \delta): R \rightarrow \C$$
and we say that $(R,\varphi)$ (or $\varphi$) is normalizable if $\varphi(1)\not= 0$ and if there exists $\lambda \in \C^*$ such that $\tilde{\varphi}=\lambda \varphi$. Using the definition of Frobenius algebra in terms of coalgebras, 
the coproduct is $\Delta = (m \otimes {\rm id}_R) \circ ({\rm id}_R \otimes  \delta)=({\rm id}_R\otimes m) \circ ( \delta \otimes  {\rm id}_R)$, and we have $\tilde{\varphi}=\varphi \circ m \circ \Delta$.

The condition that $\varphi$ is normalizable is equivalent to require, in the language of  \cite[Definition 3.1]{kad}, that $R/\C$ is a strongly separable extension with Frobenius system $(\varphi, x_i,y_i)$, where $\delta(1)=\sum_ix_i \otimes y_i$. It thus follows that if $\varphi$ is normalizable, then $R$ is necesarily a separable (semisimple) algebra. Conversely, if $R$ is semisimple,  writing $R$ as a direct product of matrix algebras, one easily sees the conditions that ensure that $\varphi$ is normalizable, see \cite{mro}.

It is shown in \cite{mro} (Corollary 4.9), generalizing earlier results from \cite{ba99,ba02,dervan}, that if $(R,\varphi)$ is a finite-dimensional semisimple measured algebra with $\dim(R)\geq 4$ and $\varphi$ normalizable, then there exists $q \in \C^*$ with $q+q^{-1} \not=0$ such that 
$$\mathcal M^{A_{\rm aut}(R,\varphi)} \simeq^{\otimes} \mathcal M^{\mathcal O({\rm PSL}_q(2))}$$
The parameter $q$ is determined as follows. First consider $\lambda \in \C^*$ such that $\tilde{\varphi}=\lambda \varphi$ and choose $\mu \in \C^*$ such that $\mu^2 = \lambda \varphi(1)$. Then $q$ is any solution of the equation
$q+q^{-1}=\mu$ (recall that $\mathcal O({\rm PSL}_q(2))=\mathcal O({\rm PSL}_{-q}(2))$, so the choice of $\mu$ does not play any role).

As an example, for $(\C^n,\varphi_n)$ as above (and $n \geq 4$), $\varphi_n$ is normalizable with the corresponding $\lambda$ equal to $1$, and $q$ is any solution of the equation $q+q^{-1}=\sqrt{n}$. 

\begin{theorem}\label{thm:cohomAut}
 Let $(R,\varphi)$ be a finite-dimensional semisimple measured algebra with $\dim(R)\geq 4$ and $\varphi$ normalizable. Assume that $A_{\rm aut}(R,\varphi)$ is cosemisimple. Then we have
\begin{equation*}
  H^n_{b}(A_{\rm aut}(R,\varphi)) \simeq \begin{cases}
                         0 & \text{if $n \not=0,3$} \\
			 \C & \text{if $n=0,3$}
                        \end{cases}
 \end{equation*}
and
 ${\rm cd}(A_{\rm aut}(R,\varphi)) \leq {\rm cd}_{\rm GS}(A_{\rm aut}(R, \varphi))=3$, with equality if $\varphi$ is a trace. In particular we have  ${\rm cd}(A_s(n))= 3= {\rm cd}_{\rm GS}(A_s(n))$ for any $n\geq 4$.
\end{theorem}

\begin{proof}
 The proof follows immediately from the combination of the above monoidal equivalence, the monoidal invariance of Gerstenhaber-Schack cohomology, Theorem \ref{cdb+e}, Theorem \ref{cobib+e}, Theorem \ref{gsh}, and Corollary \ref{cor:partansques2} ($A_{\rm aut}(R,\varphi)$ being of Kac type when $\varphi$ is a trace).
\end{proof}

Note that the length $3$ resolution of the trivial Yetter-Drinfeld module over $\mathcal O({\rm PSL}_q(2))$
by relative projective Yetter-Drinfeld modules considered in the previous subsection (see the proof of Theorem \ref{cobib+e}) transports to a  length $3$ resolution of the trivial Yetter-Drinfeld module over $A_{\rm aut}(R,\varphi)$ by relative projective Yetter-Drinfeld modules (see Theorem 4.1 in \cite{bic}), and in particular this yields 
a length $3$ projective resolution of the trivial module over $A_{\rm aut}(R,\varphi)$. We have not been able to write down this resolution explicitly enough to compute Hochschild cohomology groups and show that one always has ${\rm cd}(A_{\rm aut}(R,\varphi))=3$. We believe that this is true however.

\begin{rem}
It follows that  the $L^2$-Betti numbers (\cite{ky08})
$\beta_k^{(2)}(A_s(n))$ vanish for $k \geq 4$, and we have as well $\beta_0^{(2)}(A_s(n))=0$ by \cite{kye11}.
\end{rem}

\section{Hopf algebras with a projection}

It is natural to ask whether similar results to those of Section 2 hold for Gerstenhaber-Schack cohomological dimension. A positive answer to Question \ref{ques2} would of course provide an affirmative answer. So far, our only positive result in this direction is the following one, in the setting of Hopf algebras with a projection \cite{rad85, maj}.

\begin{proposition}
 Let $B \subset A$ be a Hopf subalgebra. Assume that there exists a Hopf algebra map $\pi : A \rightarrow B$ such that $\pi_{|B}={\rm id}_B$ and that $A$ is cosemisimple. Then we have ${\rm cd}_{\rm GS}(B) \leq {\rm cd}_{\rm GS}(A)$. 
\end{proposition}

\begin{proof}
 The inclusion $B \subset A$ together with the Hopf algebra map $\pi : A \rightarrow B$ induce a vector space preserving linear exact tensor functor 
$$F : \yd_A^A \longrightarrow \yd_B^B$$
where if $V$ is Yetter-Drinfeld module over $A$, then $F(V)=V$ as a vector space, the $B$-module structure is the restriction of that of $A$, and the $B$-comodule structure is given by $({\rm id}_V \otimes \pi)\alpha$, where $\alpha$ is the original co-action of $A$.
We claim that it is enough to show that $F$ sends (relative) projective Yetter-Drinfeld modules over $A$ to (relative) projective Yetter-Drinfeld modules over $B$. Indeed, if we have a length $n$ resolution of the trivial  Yetter-Drinfeld module over $A$ by (relative) projectives, the functor $F$ will transform it into a a length $n$ resolution of the trivial  Yetter-Drinfeld module over $B$ by (relative) projectives, and hence by Proposition \ref{resorelative}, we have ${\rm cd}_{\rm GS}(B) \leq {\rm cd}_{\rm GS}(A)$. 

As usual, put $R= {^{{\rm co} B}\!A}=\{a \in A \ | \ \pi(a_{(1)})\otimes a_{(2)}=1\otimes a\}$. This is a subalgebra of $A$ and we have $({\rm id} \otimes \pi)\Delta(R)\subset R \otimes B$, which endows $R$ with a right $B$-comodule structure. For any $a\in A$, we have  $a_{(2)}\pi S^{-1}(a_{(1)}) \in R$ (since $A$ is cosemisimple, its antipode is bijective), and thus
 we have a linear isomorphism \cite{rad85,maj}
\begin{align*}
 A &\longrightarrow R \otimes B \\
a & \longmapsto a_{(3)}\pi S^{-1}(a_{(2)}) \otimes \pi(a_{(1)})
\end{align*}
whose inverse is the restriction of the multiplication of $A$. Let $V$ be a right $A$-comodule: it also has a right $B$-comodule structure obtained using the projection $\pi : A \rightarrow B$, that we denote $V_\pi$. Consider now the map  
\begin{align*}
F(V \boxtimes A) &\longrightarrow (V_\pi \otimes R) \boxtimes B \\
v \otimes a & \longmapsto v \otimes a_{(3)}\pi S^{-1}(a_{(2)}) \otimes \pi(a_{(1)})
\end{align*}
This is an isomorphism by the previous considerations, and it is a direct verification to check that it is a morphism of Yetter-Drinfeld modules over $B$. Hence the functor $F$ sends free  Yetter-Drinfeld modules over $A$ to free Yetter-Drinfeld modules over $B$, and since it is additive, it sends, by Proposition \ref{caracrelproj},  projective Yetter-Drinfeld modules over $A$ to projective Yetter-Drinfeld modules over $B$. This concludes the proof.
\end{proof}

As an illustration, consider the hyperoctahedral Hopf algebra $A_h(n)$ \cite{bbc}. This is the algebra presented by generators  $a_{ij}$, $1 \leq i,j\leq n$, submitted to the relations
$$\sum_{l=1}^na_{li}^2=1= \sum_{l=1}^na_{il}^2, \ a_{ik}a_{ij}= 0 =a_{ji}a_{ki}\ {\rm if} \  j\not=k, \ 1\leq i,j,k\leq n$$
Its Hopf algebra structure is given by the same formulas as those for $A_s(n)$. There exist Hopf algebra
maps $i: A_s(n) \rightarrow A_h(n)$, $x_{ij} \mapsto a_{ij}^2$, $\pi : A_h(n) \rightarrow A_s(n)$, $a_{ij} \mapsto x_{ij}$, such that $\pi i={\rm id}$. Hence we deduce from the previous proposition that 
${\rm cd}_{\rm GS}(A_h(n))\geq {\rm cd}_{\rm GS}(A_s(n))$, and hence by Theorem \ref{thm:cohomAut}, if $n \geq 4$, we have ${\rm cd}_{\rm GS}(A_h(n))\geq {\rm cd}_{\rm GS}(A_s(n))=3$ (since $A_h(n)$ is cosemisimple of Kac type, this could be deduced as well from the combination of Proposition \ref{sub} and Corollary \ref{cor:partansques2}).

\end{document}